\documentclass[11pt]{article}
\usepackage[utf8x]{inputenc}
\usepackage{thm-restate}
\usepackage[letterpaper,top=1in,bottom=1in,left=1in,right=1in]{geometry}
\usepackage[utf8x]{inputenc}
\usepackage{thm-restate}
\usepackage[letterpaper,top=1in,bottom=1in,left=1in,right=1in]{geometry}
\usepackage{amsmath}
\usepackage{amssymb}
\usepackage{amsthm}
\usepackage{hyperref}
\usepackage[capitalize]{cleveref} 
\usepackage{fancyhdr}
\usepackage{graphicx}
\usepackage{xcolor}
\usepackage[shortlabels]{enumitem}

\newcommand{\set}[1]{\left\{#1\right\}}

\usepackage{xspace}
\newcommand{\Erdos}{Erd\H{o}s\xspace}
\newcommand{\Szekely}{Sz\'ekely\xspace}
\newcommand{\Lovasz}{Lov\'asz\xspace}
\newcommand{\Chvatal}{Chv\'atal\xspace}

\newtheorem{lemma}{Lemma}
\newtheorem{definition}{Definition}
\newtheorem{conjecture}{Conjecture}

\newtheorem{proposition}{Proposition}
\newtheorem{theorem}{Theorem}

\newtheorem{corollary}{Corollary}

\title{Small and Large Quasi-Kernels in General Digraphs}

\author{Daniel Rutschmann\thanks{
Current affiliation: Institute of Science and Technology Austria (ISTA), Klosterneuburg, Austria
}\\
\href{mailto:rutschmann@proton.me}{rutschmann@proton.me}\\
\url{https://orcid.org/0009-0005-6838-2628}}

\begin{document}

\thispagestyle{empty}

\maketitle

\begin{abstract}
    In a digraph $D$, a set $Q$ of vertices is a \emph{quasi-kernel} if $Q$ is an independent set
    and every vertex in $D$ can be reached from $Q$ via at most two edges.
    The \emph{small quasi-kernel conjecture} by \Erdos and \Szekely predicts
    that every source-free $n$-vertex digraph contains a quasi-kernel of size at most $n - n/2$.
    \Chvatal and \Lovasz' 1974 work implies a bound of $n - \sqrt{n}+1$.
    In 2024, Sprio shows the bound $n - \frac{1}{4} \sqrt{n \log_2 n}$.
    We show that there always is a quasi-kernel of size at most $n - \frac{1}{51}n^{2/3} / \ln^{2/3}(n)$.

    For the related \emph{large quasi-kernel conjecture} by Spiro,
    we show that there always is a quasi-kernel $Q$ with $|N^+[Q]| \ge 0.2917 \cdot n^{2/3} / \ln^{2/3}(n)$.
\end{abstract}


\paragraph{Funding} This research was supported by the Carlsberg Foundation Young Researcher Fellowship CF21-0302 -- ``Graph Algorithms with Geometric Applications''.

\paragraph{Acknowledgement} The results in this paper were developed during my three-month visit
to the University of Waterloo in Spring 2025, which was part of my PhD program at the Technical University of Denmark.
I would like to thank Sophie Spirkl for putting this problem on my radar
and for many helpful discussions and ideas.

\paragraph{Declarations}
This is a preprint version submitted to the Journal of Graph Theory in October 2025,
and will be updated in accordance with the journal's archiving policy.
All lemmas, theorems, and results were obtained by humans only.
No large language models were used, not even for spell checking.

\section{Introduction}

A \emph{digraph} $D = (V(D), E(D))$ is a finite simple directed graph. For a set $S$ of vertices, let
\[
N^+[S] = S \cup \set{v \ |\ \exists s \in S: sv \in E(D)} \text{ and } N^+(S) = N[S] \setminus S. 
\]
We similarly define $N^-[S]$ and $N^-(S)$.
A \emph{kernel} is an independent set $K$ with $N^+[K] = V(G)$.
Every acyclic digraph has a kernel, but the directed three-cycle does not.
A \emph{quasi-kernel} is an independent set $Q$ with $N^+[N^+[Q]] = V(G)$.
In 1974, \Chvatal and \Lovasz~\cite{chvatal_every_1974} show that every digraph has a quasi-kernel.
In 1976, \Erdos and \Szekely raise the following conjecture.
\begin{conjecture}[Small Quasi-Kernel Conjecture~\cite{erdos_two_2010}]
    Every source-free $n$-vertex digraph $D$ has a quasi-kernel $Q$ of size at most $n / 2$.
\end{conjecture}

This conjecture has remained wide open, despite many attempts at resolving it.
There are many results for special classes of digraphs:
Heard and Huang~\cite{heard_disjoint_2008} show the conjecture for semicomplete multipartite digraphs,
quasi-transitive digraphs, and locally semicomplete digraphs;
Kostochka, Luo and Shan~\cite{kostochka_towards_2022} for all
orientations of graphs of chromatic number $4$, or more generally, all digraphs whose vertices
can be partitioned into two kernel perfect digraphs;
Ai, Gerke, Gutin, Yeo and Zhou~\cite{ai_results_2023} for anti-claw-free digraphs; and
Langlois, Meunier, Rizzi, Vialette and Zhou~\cite{langlois_quasi-kernels_2025} for split digraphs.
\Erdos, Gy\H{o}ri, Mezei, Salia and Tyomkyn~\cite{erdhos_small_2025}
describe the history of the small quasi-kernel conjecture and related problems in a recent survey.

However, for general (source-free) digraphs, not much is known.
Spiro~\cite{spiro_generalized_2024} shows $|Q| \le n - \sqrt{n}$ and
$|Q| \le n - \frac{1}{4} \sqrt{n \log_2 n}$.
Frederickson and Yepremyan (private communication) note that
Spiro's proof can be modified to yield the bound $|Q| \le n - \frac{1}{2} \sqrt{n \log_2 n}$.
We note that \Chvatal and \Lovasz'~\cite{chvatal_every_1974} proof for the existence
of quasi-kernels, which is 2 years older than the small quasi-kernel conjecture,
already yields $|Q| \le n - \sqrt{n} + 1$ (we forward reference \cref{lemm:small_root_n}).
So, 50 years later, despite signifiant efforts, this bound has only been improved by a $\sqrt{\log n}$-factor.
Even just proving the bound $|Q| \le n - n^{1/2 + \varepsilon}$ for some $\varepsilon > 0$ seems difficult.

Spiro~\cite{spiro_generalized_2024} raises the following related conjecture,
which seems to capture the main difficulty in proving the small quasi-kernel conjecture.
\begin{conjecture}[Large Quasi-Kernel Conjecture~\cite{spiro_generalized_2024}]
    Every $n$-vertex digraph $D$ has a quasi-kernel $Q$ with $|Q^+[N]| \ge n/2$.
\end{conjecture}
Spiro proves that the small quasi-kernel conjecture implies the large quasi-kernel conjecture,
and shows the bound $|N^+[Q]| \ge n^{1/3}$.
Conversely, Ai, Liu and Peng~\cite{ai_variable_2025} show that the large quasi-kernel conjecture 
implies a weaker version of the small quasi-kernel conjecture, namely $|Q| \le n - n/3$.
However, this reduction does not apply to sublinear bounds,
such as the $|N^+[Q]| \ge n^{1/3}$ bound of Spiro.

\paragraph{Our contribution} We prove the following for the large and small quasi-kernel conjecture:
\begin{restatable}[Large Quasi-Kernel Conjecture]{theorem}{theoLarge} \label{theo:main_large}
    For $n \ge 2$, every $n$-vertex digraph has a quasi-kernel $Q$ with $|N^+[Q]| \ge 0.2917 \,\frac{n^{2/3}}{\ln^{2/3}n}$.
\end{restatable}
We show a reduction from the small quasi-kernel conjecture to the large quasi-kernel conjecture for sublinear bounds.
\begin{restatable}{theorem}{theoReduction} \label{theo:main_reduction}
    Let $f : \mathbb{N} \to [0, \infty)$ be a non-decreasing function such that
    every $n'$-vertex digraph has a quasi-kernel $Q$ with $|N^+[Q]| \ge f(n')$.
    Then every source-free $n$-vertex digraph has a quasi-kernel $Q$
    of size $|Q| \le n - \frac{\sqrt{2}-1}{3 \sqrt{2}} f(\lceil (2-\sqrt{2}) \cdot n \rceil)$.
\end{restatable}

By applying \cref{theo:main_reduction} to \cref{theo:main_large},
we get the following bound for the small quasi-kernel conjecture.

\begin{theorem}[Small Quasi-Kernel Conjecture] \label{theo:main_small}
    Every source-free $n$-vertex digraph has a quasi-kernel $Q$ of size $|Q| \le n - \frac{n^{2/3}}{51 \ln^{2/3} n}$.
\end{theorem}

Thus, we achieve an exponent of $2/3$. For the small quasi-kernel conjecture on general digraphs,
this is the first result with an exponent $> 1/2$.

Our main technical contribution is the following unorthodox strategy for the large quasi-kernel conjecture.
Rather than fixing $n$ and looking at digraphs on $n$ vertices,
we instead fix an integer $A$ and consider all digraphs with $|N^+[Q]| \le A$.
We then study the number of vertices that lie in some quasi-kernel, as a function of $n$.
This function is trivially bounded by $n$, but it also grows rapidly.
Hence, $n$ cannot be too big compared to $A$.

\section{Preliminaries}

In a digraph $D$,
the \emph{in-degree} of a vertex $u$ is $|N^-(u)|$, and we call the vertices in $N^-(u)$ the \emph{in-neighbors} of $u$.
Similarly, $|N^+(u)|$ is the \emph{out-degree} of $u$, and the vertices in $N^+(u)$ are \emph{out-neighbors}.

\paragraph{CL steps}
Let $D$ be a non-empty digraph.
\Chvatal and \Lovasz~\cite{chvatal_every_1974}
introduce the following process for finding a quasi-kernel in $D$.
\begin{itemize}
    \item (CL step) Pick $u \in V(D)$ and form the subdigraph $D' := D - N^+[u]$.
    \item (CL unstep) Let $Q'$ be any quasi-kernel of $D'$.
    If $u \in N^+[Q']$, then $Q'$ is a quasi-kernel of $D$, and otherwise, $Q' \cup \{u\}$ is a quasi-kernel of $D$.
\end{itemize}
We will use the term \emph{CL (un-)step} both as a noun and as a verb.
\Chvatal and \Lovasz~\cite{chvatal_every_1974}
CL~step at vertices $u_1, \dots, u_k$ until all vertices have been deleted.
Then, starting from $Q' = \emptyset$, they CL unstep at $u_k, \dots, u_1$,
which yields a quasi-kernel of the original digraph.
This shows that every digraph has a quasi-kernel.

In general, if we perform some sequence of CL steps,
then we should do the CL unsteps in reverse order.
The next two Lemmas illustrate baseline bounds for the large and small quasi-kernel conjecture that can be obtained
directly from \Chvatal and \Lovasz' methods.
First, if we repeatedly CL~step at a vertex of minimum out-degree,
we get the following bound for the large quasi-kernel conjecture:

\begin{lemma} \label{lemm:large_root_n}
    Every $n$-vertex digraph $D$ has a quasi-kernel $Q$ with $|N^+[Q]| \ge \sqrt{n}$.
\end{lemma}
\begin{proof}
    Repeatedly CL~step at a vertex of minimum out-degree until no vertices remain.
    Then, perform all CL unsteps to get a quasi-kernel $Q$ of $D$.
    If we ever encounter a digraph $D'$ with minimum out-degree $\ge \sqrt{n} - 1$, then
    $|N^+[Q]| \ge \sqrt{n}$ since $Q$ contains a vertex of $D'$.

    If this doesn't happen, then every CL step deletes at most $\sqrt{n}$ vertices.
    Hence, we perform at least $\sqrt{n}$ CL steps.
    Observe that, if we CL unstep at a vertex $u$, then we always have $u \in N^+[Q]$.
    Indeed, either $u$ was already in $N^+[Q']$,
    or, if not, then we add $u$ to $Q'$.
    Hence, after $\sqrt{n}$ CL~unsteps, $|N^+[Q]| \ge \sqrt{n}$.
\end{proof}

If we always CL~step at a non-source vertex, we get a similar bound for the small quasi-kernel conjecture.

\begin{lemma} \label{lemm:small_root_n}
    Every source-free $n$-vertex digraph $D$ has a quasi-kernel $Q$ with $|Q| \le n - \sqrt{n} + 1$.
\end{lemma}
\begin{proof}
    Observe that, if we CL step at a vertex $u$, then the vertices in $N^+(u)$
    get deleted and they cannot end up in $Q$ when we CL unstep.
    We call these \emph{good} vertices.
    Our goal is to find a sequence of CL steps with at least $\sqrt{n} - 1$ good vertices.
    Then, the CL unsteps yield a quasi-kernel $Q$ with $|Q| \le n - \sqrt{n} + 1$.

    If $D$ contains a vertex $u$ with out-degree $\ge \sqrt{n} - 1$, then we
    CL step at $u$, which immediately yields $\sqrt{n} - 1$ good vertices.
    Suppose therefore that $G$ has maximum out-degree $\le \sqrt{n} - 2$.
    Repeatedly CL step at a non-source vertex until the digraph contains only sources.
    Let $k$ be the number of good vertices.
    A CL~step at a non-source vertex creates at least one good vertex,
    so we performed at most $k$ CL steps, and delete at most $2k$ vertices in total.
    On the other hand, a CL step can only create sources at the out-neighbors of
    good vertices. Since $D$ has maximum out-degree $\le \sqrt{n} - 2$,
    we create at most $k \cdot (\sqrt{n} - 2)$ sources in total.
    But after all CL steps, the remaining vertices are all sources. Hence, $n - 2k \le k \cdot (\sqrt{n} - 2)$.
    This implies $k \ge \sqrt{n}-1$.
\end{proof}

We generalize CL steps in the following way:
If $S$ is an independent set, we may perform CL steps or CL unsteps at all vertices in $S$.
We may process the vertices in $S$ in any order, this does not affect the result. This yields the following operations:
\begin{itemize}
    \item (CL step at $S$) Form the subdigraph $D' := D - N^+[S]$.
    \item (CL unstep at $S$) Let $Q'$ be any quasi-kernel of $D'$. Put $S' = \{u \in S \ | \ u \notin N^-[D']\}$,
        then $Q' \cup S'$ is a quasi-kernel of $D$.
\end{itemize}

Finally, we introduce a few more notions that are useful for our proofs.

\begin{definition}
    A set $S$ of vertices is \emph{extendable} if there is a quasi-kernel $Q \supseteq S$.
    A vertex $u$ is extendable if $\{u\}$ is extendable.
\end{definition}
Any quasikernel is an extendable set,
and every extendable set is an independent set.
Deciding whether a vertex is extendable is an NP-complete problem~\cite{croitoru_note_2015}.

\begin{definition}
    A CL subgraph of $D$ is a digraph obtained via zero or more CL steps on $D$.
\end{definition}
If $D' \subseteq D$ is a CL subgraph, then any extendable set in $D'$ is also an extendable set in $D$.
If $|N^+[Q]| \le A$ holds for every quasi-kernel $Q$ in $D$, then $|N^+[S]| \le A$ also
holds for every extendable set $S$ in some CL subgraph $D'$ of $D$. In particular, a CL step at an extendable set (or extendable vertex) deletes at most $A$ vertices.

\section{Large Quasi-Kernel Conjecture}

We prove the following bound for the large quasi-kernel conjecture.

\theoLarge*

To this end, we fix a positive integer $A$, later to be set to $\lfloor0.2917 \cdot n^{2/3} / \ln^{2/3}(n)\rfloor$,
and study the family of all digraphs $D$ with $|N^+[Q]| \le A$ for all quasi-kernels $Q$ of $D$.
Our strategy is to first find many extendable vertices,
and then find many edges between extendable vertices.
If this reveals an extendable vertex with more than $A$ out-neighbors,
then $|N^+[Q]| > A$, a contradiction which shows that no $n$-vertex digraphs has $N^+[Q] \le A$ for all quasi-kernels $Q$.

\begin{lemma} \label{lemm:out_edges}
    Let $D$ be a $n$-vertex digraph with $|N^+[Q]| \le A$ for all quasi-kernels $Q$ of $D$.
    Then
    \[
        |N^-(Q) \setminus N^+(Q)| \ge n - A \quad \text{ for every \emph{maximal} quasi-kernel $Q$ of $D$.}
    \]
\end{lemma}
\begin{proof}
Since $Q$ is maximal, we have $N^+(Q) \cup Q \cup N^-(Q) = V(D)$.
Hence, $|N^-(Q) \setminus N^+(Q)| = n - |N^+[Q]| \ge n - A$.
\end{proof}

\begin{lemma} \label{lemm:lemma_one}
    Let $D$ be a $n$-vertex digraph with $|N^+[Q]| \le A$ for all quasi-kernels $Q$ of $D$.
    Let $X \le n$ be an integer and put $k = \lfloor\frac{X}{A}\rfloor$.
    Then, there is a quasi-kernel $Q$ such that $N^+[Q]$ contains $k$ disjoint non-empty
    extendable sets $S_1, \dots, S_k$ with $|N^-(S_i)| \ge n - X$.
\end{lemma}
\begin{proof}
    We do induction on $X$. For $X < A$ the Lemma holds since $k = 0$.
    Suppose now that $X \ge A$. Let $S_k$ be any maximal quasi-kernel of $D$,
    then $|N^-(S_k)| \ge n - A$ by \cref{lemm:out_edges}. Moreover, $S_k$ is an independent set.
    We CL~step at $S_k$ to get digraph $D' \subseteq D$.
    Then $n' = |V(D')| \ge n - |N^+[S_k]| \ge n - A$, so we put $X' = X - A$.
    By induction, there is a quasi-kernel $Q'$ of $D'$ that contains
    $D'$-extendable sets $S_1, \dots, S_{k-1}$ with $|N_{D'}^-(S_i)| \ge n' - X' \ge n - X$.
    Since $D'$ is a CL subgraph of $D$, these sets are also $D$-extendable.
    Starting from $Q'$, we CL unstep at $S_k$ to get a quasi-kernel $Q \supseteq Q'$ of $D$.
    Then $N^+[Q]$ contains not only $S_1, \dots, S_{k-1}$, but also $S_k$.
\end{proof}

We can repeatedly apply \cref{lemm:lemma_one} and CL step at the resulting quasi-kernel $Q$.
This yields many sets $S_i$ that are both extendable and incident to many edges.
The next two lemmas pursue this approach, and allow us to prove the weaker bound $A \ge \frac{1}{3} n^{3/5}$.

\begin{lemma} \label{lemm:weak_extendable_vertices}
    Let $n \ge 1$.
    Let $D$ be a $n$-vertex digraph with $|N^+[Q]| \le A$ for all quasi-kernels $Q$ of $D$.
    Then $D$ contains (at least) $1 + \frac{1}{2} \frac{n}{A} \Big(\frac{n}{A} - 1 \Big) \ge \frac{n^2}{3A^2}$ extendable vertices.
\end{lemma}
\begin{proof}
    We do induction on $n$. The Lemma holds for $n \le A$ as $D$ contains at least one extendable vertex.
    Suppose now that $n \ge A$.
    Put $X = n$ and $k = \lfloor \frac{X}{A} \rfloor \ge \frac{n}{A} - 1$.
    Apply \cref{lemm:lemma_one} to $D$ to obtains a quasi-kernel $Q$
    such that $N^+[Q]$ contains $k$ disjoint non-empty extendable sets.
    In particular, $N^+[Q]$ contains (at least) $k$ extendable vertices.
    CL step at $Q$, which deletes $N^+[Q]$ from $D$.
    Let $D'$ be the resulting digraph and $n' = |V(D')| \ge n-A$.
    By induction, $D'$ contains $1 + \frac{1}{2} \frac{n'}{A} (\frac{n'}{A} - 1)$ extendable vertices.
    Hence $D$ contains
    \[
        1 + \frac{1}{2} \frac{n'}{A} \Big(\frac{n'}{A} - 1\Big) + k \ge 1 + \frac{1}{2} \Big(\frac{n}{A} - 1\Big) \Big(\frac{n}{A} - 2\Big) + \frac{n}{A} - 1 = 1 + \frac{1}{2} \Big(\frac{n}{A} - 1\Big) \frac{n}{A}
    \]
    extendable vertices. Finally, $1 + \frac{1}{2} x (x-1) \ge \frac{1}{3} x^2$ for all $x \in \mathbb{R}$.
\end{proof}

\begin{lemma} \label{lemm:weak_extendable_in}
    Let $n \ge 1$.
    Let $D$ be a $n$-vertex digraph with $|N^+[Q]| \le A$ for all quasi-kernels $Q$ of $D$.
    Then $D$ contains an extendable vertex with $\ge \frac{1}{81} n^3 / A^4 - \frac{1}{6} n / A $ extendable out-neighbors.
\end{lemma}
\begin{proof}
    Put $X = \lfloor\frac{2}{27} n^2 / A^2\rfloor$ and $k = \lfloor \frac{X}{A} \rfloor \ge \frac{2}{27} n^2 / A^3 - 1$.
    Set $D' = D$ and repeat the following $\lceil \frac{1}{3}n / A \rceil$ times:
    Apply \cref{lemm:lemma_one} to $D'$ to obtains a quasi-kernel $Q$
    such that $N^+[Q]$ contains $k$ disjoint non-empty sets $S_1, \dots, S_k$ with $|N^-(S_i)| \ge n - X$.
    Then, in particular, these $k$ sets have edges from all but $\le X$ vertices.
    Finally, if this is not the last repetition, CL~step at $Q_0$, which deletes $N^+[Q_0]$ from $D'$.

    Let $D''$ be the digraph that remains after all repetitions.
    Each repetition deletes at most $A$ vertices, and the last repetition doesn't delete anything, thus $|V(D'')| \ge \frac{2n}{3}$.
    Let $T$ be the set of extendable vertices of $D''$.
    By \cref{lemm:weak_extendable_vertices}, we have $|T| \ge \frac{4}{27} n^2/A^2 \ge 2X$.
    Hence, in $D$, each set $S_i$ has edges from at least half of vertices in $T$.
    Across all repetitions, there are $\ge k \cdot \frac{1}{3}n / A \ge \frac{2}{81} n^3 / A^4 - \frac{1}{3}n / A$ sets $S_i$.
    Hence, there is a vertex in $T$ with at least $\frac{1}{81} n^3 / A^4 - \frac{1}{6} n / A$ extendable out-neighbors
    in $D$.
\end{proof}

\begin{corollary}
    Let $n \ge 1$.
    Let $D$ be a $n$-vertex digraph with $|N^+[Q]| \le A$ for all quasi-kernels $Q$ of $D$.
    Then $A \ge \frac{1}{3} n^{3/5}$.
\end{corollary}
\begin{proof}
\cref{lemm:weak_extendable_in} yields an extendable vertex $u$ with $\ge \frac{1}{81} n^3 / A^4 - \frac{1}{6} n / A $ in-neighbors. Let $Q \supseteq \{u\}$ be any quasi-kernel. If $A < \frac{1}{3} n^{3/5}$ then
\[
   |N^+(Q)| \ge \frac{1}{81} n^3 / A^4 - \frac{1}{6} n / A \ge \frac{1}{3} n^{6/5} / A - \frac{1}{6} n / A
   \ge \frac{1}{6} n^{6/5} / A \ge \frac{1}{2} n^{3/5} > A. 
\]
This contradicts the assumption $|N^+[Q]| \le A$.
\end{proof}

\subsection{Proving a Stronger Bound}
Previously, we used \cref{lemm:weak_extendable_vertices} to prove \cref{lemm:weak_extendable_in},
and used that to prove the bound $A \ge \frac{1}{4} n^{3/5}$.
To obtain a stronger bound, we will bootstrap this approach: use \cref{lemm:weak_extendable_in}
to prove a stronger version of \cref{lemm:weak_extendable_vertices}, which in turn implies a stronger
version of \cref{lemm:weak_extendable_in}.

\begin{proposition} \label{prop:recursive_bounds}
    Fix a positive integer $A$. For every positive integer $n$, let $f(n)$ and $g(n)$, respectively,
    be the largest integer such that every digraph $D$ on at least $n$ vertices
    with $|N^+[Q]| \le A$ for all quasi-kernels $Q$ of $D$ contains
    \begin{itemize}
        \item at least $f(n)$ extendable vertices, and
        \item an extendable vertices with at least $g(n)$ extendable in-neighbors.
    \end{itemize}
    If there are no such digraphs, put $f(n) = g(n) = \infty$.
    The functions $f, g : \mathbb{N} \to \mathbb{N}_0 \cup \{\infty\}$ satisfy:
    \begin{enumerate}[(1)]
        \item $f, g$ are non-decreasing, $f(1) = 1$ and $g(1) = 0$.
        \item $f(n) \ge \frac{1}{3} n^2 / A^2$ for $n \ge 1$.
        \item $f(n) \ge f(n-A) + g(n)$ for $n > A$.
        \item $f(n) \ge \frac{n}{A \ln n} \cdot g(\lceil(1-1/\ln n) \cdot n \rceil)$ for $n > \max(6, A)$.
        \item $g(n) \ge \frac{n}{2 A \ln n} \cdot \Big\lfloor f(\lceil(1-1/\ln n) \cdot n\rceil) \Big/ (2A) \Big\rfloor$ for $n \ge 3$.
        \item Pick an integer $k \ge 0$ and set $n' = \lceil(1 - 1 / \ln n)^{2 k} \cdot n\rceil$. If $n'^{\,2} / (4 A^3 (\ln n')^2) \ge 2$, then
        \[
            f(n) \ge \Big(\frac{n'^{\,2}}{4A^3 (\ln n')^2}\Big)^{k}\cdot (f(n') - 4 A) + 4A.
        \]
    \end{enumerate}
\end{proposition}
\begin{proof}
    Since we consider digraphs on \emph{at least} $n$ vertices, $f, g$ are non-decreasing.
    Every digraph has a quasi-kernel, and hence an extendable vertex. This shows (1).
    \cref{lemm:weak_extendable_vertices} shows (2).

    Let $n$ be a positive integer and let $D$ be a digraph on at least $n$ vertices
    with $|N^+[Q]| \le A$ for all quasi-kernels $Q$ of $D$.
    We mimic the proof of \cref{lemm:weak_extendable_vertices}.
    For (3), let $u$ be an extendable vertex in $D$ with at least $g(n)$ extendable out-neighbors.
    Obtain $D'$ by performing a CL step at $u$. Since $u$ is extendable, this deletes at most $A$
    vertices, so $D'$ contains at least $f(n - A)$ extendable vertices. On the other hand,
    this deletes at least $g(n)$ extendable vertices, so $D$ contains at least $f(n - A) + g(n)$
    extendable vertices. This shows (3).
    For (4), put $k = \lfloor \frac{n}{A \ln n} \rfloor$ and apply (3) $k+1$ times;
    this is possible since $n > \max(6, A)$ implies $n - (k+1) A > 0$.
    \[
        f(n) \ge f(n - (k+1) A) + g(n) + \dots + g(n - k A) \ge (k+1) \cdot g(n - k A)
    \]
    Since $k+1 \ge \frac{n}{A \ln n}$ and $n - k A \ge n - \lfloor \frac{n}{\ln n} \rfloor = \lceil(1-1/\ln n) \cdot n \rceil$, this shows (4).

    For (5), we mimic the proof of \cref{lemm:weak_extendable_in}.
    Put $X = \lfloor f(\lceil(1 - \frac{1}{\ln(n)}) \cdot n \rceil) / 2 \rfloor$
    and $k = \lfloor \frac{X}{A} \rfloor$.
    This is well-defined since $n \ge 3$.
    Set $D' = D$ and repeat the following $\lceil \frac{n}{A \ln n} \rceil$ times:
    Apply \cref{lemm:lemma_one} to $D'$ to obtains a quasi-kernel $Q$
    such that $N^+[Q]$ contains $k$ disjoint non-empty sets $S_1, \dots, S_k$ with $|N^-(S_i)| \ge n - X$.
    Then, in particular, these $k$ sets have edges from all but $X$ of the remaining vertices.
    Finally, if this is not the last repetition, perform a CL step at $Q$, which deletes $N^+[Q]$ from $D'$.

    After all repetitions, let $Y$ denote the extendable vertices in the remaining graph.
    We perform $(\lceil \frac{n}{A \ln n} \rceil - 1)$ CL steps,
    each of which deletes at most $A$ vertices. The number of vertices that remain is at least
    \[
        n - A \cdot \Big(\Big\lceil \frac{n}{A \ln n} \Big\rceil - 1\Big) \ge n - A \cdot \Big\lfloor\frac{n}{A \ln n} \Big\rfloor \ge n - \Big\lfloor\frac{n}{\ln n} \Big\rfloor = \Big\lceil\Big(1 - \frac{1}{\ln n}\Big) \cdot n\Big\rceil,
    \]
    so $|Y| \ge f(\lceil(1 - \frac{1}{\ln n}) \cdot n\rceil)$.
    By our choice of $X$, each $S_i$ has edges from at least half the vertex in $Y$.
    Across all repetitions, there are $\lceil \frac{n}{A \ln n} \rceil \cdot k \ge \frac{n}{A \ln n} \cdot k$
    different sets $S_i$. Hence, one vertex in $Y$ has at least
    $\frac{1}{2} \cdot \frac{n}{A \ln n} \cdot k$ extendable out-neighbors. This shows (5).

    Finally, (6) follows from (4) and (5) via induction over $k$: The statement is trivial for $k=0$, so suppose $k \ge 1$.
    We apply (4) followed by (5) to get
    \[
        f(n) \ge \frac{n'^2}{2 A^2 (\ln n')^2} \cdot \Big\lfloor f(\lceil(1-1/\ln n)^2 \cdot n\rceil) \Big/ (2A) \Big\rfloor \ge \frac{n'^2}{4 A^3 (\ln n')^2} \cdot \Big(f(\lceil(1-1/\ln n)^2 \cdot n\rceil) - 2A\Big).
    \]
    By induction, this is at least
    \[
        \frac{n'^{\,2}}{4 A^3 (\ln n')^2} \cdot \Big(\Big(\frac{n'^{\,2}}{4A^3 (\ln n')^2}\Big)^{k-1} (f(n') - 4A) + 2A \Big) = \Big(\frac{n'^{\,2}}{4A^3 (\ln n')^2}\Big)^{k} (f(n') - 4A) + \frac{n'^{\,2}}{4 A^3 (\ln n')^2} \cdot 2A.
    \]
    Since we assume $\frac{n'^{\,2}}{4 A^3 (\ln n')^2} \ge 2$, this shows (6).
\end{proof}

\noindent
We now prove \cref{theo:main_large}.
Let $n \ge 2$ and suppose that $A \le 0.2917 \frac{n^{2/3}}{\ln^{2/3}(n)} \le 2^{-2/3} e^{-4 \sqrt{3} / 9}\, \frac{n^{2/3}}{\ln^{2/3}(n)}$.
Are there any $n$-vertex digraphs $D$ with $|N^+[Q]| \le A$ for all quasi-kernels $Q$ of $D$?
If $n < 100$, then $A < \sqrt{n}$ and \cref{lemm:large_root_n} shows that no such digraphs exist.
Suppose therefore that $n \ge 100$.
Pick $k = \lfloor\ln(n) / \sqrt{12}\rfloor$,
then (6) defines $n' = \lceil(1 - 1/\ln(n))^{2k} \cdot n \rceil \ge n / e^{1/\sqrt{3}}$. This implies

\[
    \frac{n'^{\,2}}{4 A^3 (\ln n')^2} \ge \frac{n^2 / e^{2/\sqrt{3}}}{4 A^3 (\ln n)^2} = e^{2 / \sqrt{3}} > 3.173,
\]
so we may indeed apply (6). We get
\[
    f(n) \ge \Big(\frac{n'^{\,2}}{4 A^3(\ln n')^2}\Big)^k \cdot (f(n') - 4A) \ge e^{2 k / \sqrt{3}} \cdot (f(n') - 4A) \ge \frac{n^{1/3}}{e^{2/\sqrt{3}}} \cdot (f(n') - 4A),
\]
where the second inequality follows from $k \ge \ln(n) / \sqrt{12} - 1$.
Since $n \ge 100$, we have $A \le \frac{1}{7} n^{2/3}$.
We combine this with part (2) of \cref{prop:recursive_bounds} to show
\[
    f(n') \ge \frac{1}{3} n'^{\,2} / A^2 \ge \frac{49}{3 e^{2/\sqrt{3}}}\, n^{2/3} \ge 5 n^{2/3} \ge e^{2/\sqrt{3}} n^{2/3} + n^{2/3} \ge e^{2 / \sqrt{3}} n^{2/3} + 7 A.
\]
Plugging this into the previous equation shows $f(n) > n$, so $f(n) = \infty$. Hence,
for $n \ge 2$ and $A \le 0.2917 \frac{n^{2/3}}{\ln^{2/3}(n)}$,
there is no $n$-vertex digraph with $|N^+[Q]| \le A$ for all quasi-kernels $Q$.
This proves our main result.

\theoLarge*

\section{Reducing the Small quasi-kernel conjecture to the Large Quasi-Kernel Conjecture}

We show that any bound for the large quasi-kernel conjecture implies a similar bound for the small quasi-kernel conjecture.
\theoReduction*

To prove this theorem, we introduce a set $B$ of \emph{blue} vertices.
As an intermediate problem, we consider a variant of the large quasi-kernel conjecture 
where the goal is to maximize $|N^+[Q] \cap B|$ instead of $|N^+[Q]|$.
In other words, the quasi-kernel still needs to reach all vertices in $D$ via at most two edges,
but now we measure its ``size'' via the blue vertices.
We prove two reduction lemmas via this intermediate problem:

\begin{lemma} \label{lemm:blue_reduction}
    Let $f : \mathbb{N} \to [0, \infty)$ be a non-decreasing function such that
    every $n'$-vertex digraph has a quasi-kernel $Q$ with $|N^+[Q]| \ge f(n')$.
    Let $D$ be an $n$-vertex digraph, let $\beta > 0$
    and let $B \subseteq V(D)$ be a set with $|B| \ge \beta n$.
    Then $D$ has a quasi-kernel $Q$ with $|N^+[Q] \cap B| \ge (1-\sqrt{1-\beta}) \cdot f(\lceil(1-\sqrt{1-\beta})\, n\rceil)$.
\end{lemma}
\begin{lemma} \label{lemm:small_reduction}
    Let $g : \mathbb{N} \to [0, \infty)$ be a non-decreasing function such that
    every $n'$-vertex digraph $D$, for every $B \subseteq V(D)$ with $|B| \ge \frac{1}{2} n'$,
    has a quasi-kernel $Q$ with $|N^+[Q] \cap B| \ge g(n')$.
    Then every source-free $n$-vertex digraph has a quasikernel $Q$ with $|Q| \ge n - \frac{1}{3} g(2n)$.
\end{lemma}
Let us first show that these two Lemmas imply \cref{theo:main_reduction}:
By \cref{lemm:small_reduction}, every source-free $n$-vertex digraph has a quasikernel $Q$ with $|Q| \ge n - \frac{1}{3} g(2n)$
for a suitable function $g$.
By \cref{lemm:blue_reduction} with $\beta = \frac{1}{2}$, we may pick
\[
    g(2n) = \Big(1-\sqrt{1-\frac{1}{2}}\Big) f\Big(\Big\lceil\Big(1-\sqrt{1-\frac{1}{2}}\Big) \cdot 2n \Big\rceil\Big) = \frac{\sqrt{2} - 1}{\sqrt{2}} f(\lceil(2-\sqrt{2}) n \rceil),
\]
where $f$ is the function in \cref{theo:main_reduction}. This shows \cref{theo:main_reduction} since now
\[
    |Q| \ge n - \frac{1}{3} g(2n) = n - \frac{\sqrt{2} - 1}{3 \sqrt{2}} f(\lceil(2-\sqrt{2}) n \rceil).
\]

\subsection{Proof of \cref{lemm:blue_reduction}}
    Put $D' = D$. While $D'$ contains an independent set $S$ with $|N_{D'}^+[S] \cap B| < (1-\sqrt{1-\beta})\cdot |N^+[S]|$,
    CL~step at $S$. This deletes $N_{D'}^+[S]$, which includes at least $\sqrt{1-\beta} \cdot |N_{D'}^+[S]|$ vertices in $D'$
    that are not in $B$.
    Since $D$ contains at most $(1-\beta) n$ vertices not in $B$,
    we delete at most
    \[
        (1-\beta) n / \sqrt{1-\beta} = \sqrt{1-\beta}\cdot n
    \]
    vertices in total. In particular, $\ge (1-\sqrt{1-\beta})\cdot n$ vertices remain after all CL steps.
    Then, let $Q'$ be a quasi-kernel of $D'$ with $|N_{D'}^+[Q']| \ge f(\lceil(1-\sqrt{1-\beta})\,n\rceil)$. We have
    \[
        |N_{D'}^+[Q'] \cap B| \ge (1-\sqrt{1-\beta}) \cdot |N_{D'}^+[Q']| \ge (1-\sqrt{1-\beta}) \cdot f(\lceil(1-\sqrt{1-\beta}) n\rceil).
    \]
    Finally, repeatedly CL unstep to turn $Q'$ into a quasi-kernel $Q$ of $D$ with $|N_{D}^+[Q] \cap B| \ge |N_{D'}^+[Q'] \cap B|$.
    This shows \cref{lemm:blue_reduction}.

\subsection{Proof of \cref{lemm:small_reduction}}

We first prove a helper lemma.
\begin{lemma} \label{lemm:big_matching}
    Let $D$ be a $n$-vertex digraph with a matching $M$.
    Then $D$ has a quasi-kernel of size $\le n - |M|$.
\end{lemma}
\begin{proof}
    For each edge $uv$ in $M$ in order, CL step at $v$. If $v$ has already been deleted,
    but $u$ has not, CL step at $u$ instead. In total, we CL step at most $|M|$ times,
    and delete all vertices incident to $M$.
    Let $Q'$ be a quasi-kernel of the remaining graph. Then, $|Q'| \le n - 2 |M|$.
    CL~unstep all the way. This yields a quasi-kernel of $D$ of size $\le |Q'| + |M| \le n - |M|$.
\end{proof}

We now prove \cref{lemm:small_reduction}. Let $M$ be a maximal matching in $D$.
If $|M| \ge \frac{1}{3} g(2n)$, then \cref{lemm:small_reduction} follows from \cref{lemm:big_matching}.
Otherwise, for each vertex $u \in D$ one after the other, pick an arbitrary vertex $v \in N^-(u)$ and
``clone'' $v$ by adding a new vertex $\tilde{u}$ to $D$ with
$N^-(\tilde{u}) = N^-(v)$ and $N^+(\tilde{u}) = N^-(v)$.
(Clones of the same vertex form an independent set. Sometimes,
this is called an \emph{unbalanced blow up} of $D$.)
This yields a digraph $D'$ on $2n$ vertices.
Note that $X := V(D) - V(M)$ is an independent set, so if $u \in X$, then $v \in V(M)$.

Put $B = \{\tilde{u} \ |\ u \in V(D)\}$
and $\beta = |B| / (2n) = \frac{1}{2}$. By \cref{lemm:blue_reduction},
there is a quasi-kernel $Q'$ of $D'$ with $|N_{D'}^+[Q'] \cap B| \ge g(2n)$.
Without loss of generality, pick $Q'$ to be maximal. This ensures that $Q'$
contains all clones of a vertex plus the original vertex if it contains one of them.
In particular, $Q' \cap V(D)$ is a quasi-kernel of $D$.

Using $Q'$, we construct a small quasi-kernel $Q$ of $D$:

\begin{enumerate}[1),noitemsep,leftmargin=*]
    \item For each $u \in V(M) \cap N_{D'}^+[Q']$, add an (arbitrary) vertex from $Q' \cap V(D) \cap N_{D}^-[u]$ to $Q$. 
    \item For each $u \in V(M) - N_{D'}^+[Q']$, add a vertex from $Q' \cap V(D) \cap N_{D}^-[N_{D}^-[u]]$ to $Q$. 
    \item For each $u \in X$ with $\tilde{u} \notin N_{D'}^+[Q']$, add a vertex from $Q' \cap V(D) \cap N_{D}^-[N_{D}^-[u]]$ to $Q$.
\end{enumerate}
The first two cases together add $|V(M)|$ vertices to $Q$.
The third case adds at most ${|B - N_{D'}^+[Q']|} = n - |N_{D'}^+[Q'] \cap B|$ vertices to $Q$.
Recall that $|M| \le \frac{1}{3}g(2n)$.
Hence,
\[
    |Q| \le |V(M)| + (n - |N_{D'}^+[Q'] \cap B|) \le 2|M| + n - g(2n) \le n - \frac{1}{3} g(2n).
\]

Next, we check that $Q$ is a quasi-kernel of $D$.
In cases (1)-(3), we explicitly add a vertex in $N_D^-[N_D^-[u]]$ to $Q$. The only remaining case
is $u \in X$ with $\tilde{u} \in N_{D'}^+[Q']$. But then, $\tilde{u}$ is a clone of a vertex $v \in V(M)$ with
$v \in N_{D'}^+[Q']$. By case (1), $v \in N_D^+[Q]$, hence $u \in N_D^+[N_D^+[Q]]$ since we picked $v \in N_D^-(u)$.
This shows that $Q$ is a quasi-kernel of $D$, which concludes the proof of \cref{lemm:small_reduction}.

\subsection{Technical Details for Deriving \cref{theo:main_small} from \cref{theo:main_large}.}
There is a small technical hurdle when applying \cref{theo:main_reduction}
to \cref{theo:main_large}: The function $\frac{x^{2/3}}{\log^{2/3}(x)}$ is only non-decreasing for $x \ge e$, and
is not defined for $x = 1$. But, we can put $f(1) = f(2) = 0$ and $f(n) = 0.2917 \frac{n^{2/3}}{\log^{2/3}(n)}$ for $n \ge 3$,
so that $f$ is non-decreasing. Then, \cref{theo:main_reduction} shows \cref{theo:main_large} for, say, $n \ge 11$.
For $n=1$, \cref{theo:main_large} is vacuously true (there is no source free 1-vertex digraph),
and for $2 \le n \le 10$, \cref{theo:main_large} follows from \cref{lemm:small_root_n}.

\bibliography{zotero,references}

\end{document}